\documentclass[a4paper,12pt]{amsproc}
\title{Measure finite topology on the ring of measurable functions}
\usepackage{amsfonts, amsmath, amssymb, graphicx, mathrsfs, geometry}
\usepackage[all]{xy}
\usepackage{xcolor}
\usepackage{hyperref}
\hypersetup{
	colorlinks=true,
	linkcolor={red!60!black},
	citecolor={blue!50!black},
	urlcolor={green!30!black}
}
\newdir{ >}{!/6pt/@{ }*:(1,0.2)@_{>}*:(1,-0.2)@^{>}}

\theoremstyle{plain}
	
	\newtheorem{theorem}{Theorem}[section]
	
	\title{Measure finite topology on the rings of measurable functions}
	\usepackage{amsfonts, amsmath, amssymb, graphicx, mathrsfs}
	\usepackage[all]{xy}
	
	\newdir{ >}{!/6pt/@{ }*:(1,0.2)@_{>}*:(1,-0.2)@^{>}}
	
	\theoremstyle{plain}

	\newtheorem{lemma}[theorem]{Lemma}

	\theoremstyle{definition}
	\newtheorem{definition}[theorem]{Definition}

	\newtheorem{remark}[theorem]{Remark}
	\newtheorem{counter example}[theorem]{Counter Example}

	\newtheorem{corollary}[theorem]{Corollary}

	\newtheorem{example}[theorem]{Example}

	\numberwithin{equation}{section}
	
	\DeclareMathOperator{\cl}{cl}

	\author[S. Dey]{Soumajit Dey}	\address{Department of Pure Mathematics, University of Calcutta, 35, Ballygunge Circular Road, Kolkata 700019, West Bengal, India}	\email{deysoumajit8@gmail.com}

	\author[S. K. Acharyya]{Sudip Kumar Acharyya}	\address{Department of Pure Mathematics, University of Calcutta, 35, Ballygunge Circular Road, Kolkata
		700019, West Bengal, India}	\email{sdpacharyya@gmail.com}
	
	\author[D. Mandal]{Dhananjoy Mandal} \address{Department of Pure Mathematics, University of Calcutta, 35 , Ballygunge Circular Road, Kolkata 700019, West Bengal, India}  \email{dmandal.cu@gmail.com / dmpm@caluniv.ac.in}
	
\begin{document}
		\title [Measure finite topology on the ring of measurable functions]{Measure finite topology on the ring of measurable functions }
		\thanks {The first author extends immense gratitude and thanks to the University Grants Commission, New Delhi, for the award of research fellowship (NTA Ref. No. 211610214962).}
	\maketitle
	\begin{abstract}
		Let $\mathcal{M}(X,\mathcal{A},\mu)$ be the ring of all real-valued measurable functions constructed over a measure space $(X,\mathcal{A},\mu)$. A topology on $\mathcal{M}(X,\mathcal{A},\mu)$, called the {$F_\mu$-topology} weaker than the { $U_\mu$-topology} is introduced. It is realized that the {component}, the {quasi component} and the {path component }in this {$F_\mu$-topology} are identical. It turns out that the {$F_\mu$-topology} on $\mathcal{M}(X,\mathcal{A},\mu)$ becomes {connected} if and only if it is {path connected} if and only if $\mu$ is an {atomic measure} of a special type. It is also proved that the {$F_\mu$-topology} is {first countable} when and only when $\mu$ is a {hemifinite measure.} Finally, it is shown that the {second countability} of the {$F_\mu$-topology}  is equivalent to the {hemifiniteness} of the measure $\mu$ together with the {countable chain condition} of the {$F_\mu$-topology}.
	\end{abstract}
	\section{ Introduction}
	
	The topology of {uniform convergence} or the {$u$-topology} and the {fine-topology} or the \text{$m$-topology} on the ring $C(X)$ of all real-valued continuous functions over a Tychonoff space $X$ are extensively studied and there is a huge literature dealing with various problems related to these topics. For an introduction to these two topologies see Problem 2M and 2N in \cite{11}. Appropriate measure-theoretic analogue of these two topologies are defined on the ring $\mathcal{M}(X, \mathcal{A}, \mu)$ of all real-valued measurable functions constructed over a measure space $(X, \mathcal{A}, \mu)$, rather recently (see  \cite{4}, \cite{1},  \cite{5}, \cite{2}, \cite{3}, \cite{6} in this regard). This article aims at continuing a similar line of study. Indeed, we are strongly motivated by the following fact: the ring $C(X)$ with compact-open topology is first countable if and only if $X$ is a hemicompact space \cite{7}. Our intention is to initiate a measure-theoretic counterpart of both the compact-open topology and the hemicompactness of a topological space, in order to achieve amogest other results, an analogous fact in a different setting. For any $f \in \mathcal{M}(X, \mathcal{A}, \mu)$, $A \in \mathcal{A}$ with $\mu(A) < \infty$ and $\epsilon > 0$, we set
\[B(f, A, \epsilon) = \{ g \in \mathcal{M}(X, \mathcal{A}, \mu) : \text{there exists } A_g \in \mathcal{A} \text{ with } \mu( A_g) = 0 \text{ such that }\]\[ \sup_{x \in A\setminus A_g} |f(x) - g(x)| < \epsilon \}.\]
Then it is not hard to check that the family 
	\[
	\mathcal{S} = \{ B(f, A, \epsilon) : f \in \mathcal{M}(X, \mathcal{A}, \mu),\, A \in \mathcal{A} \text{ with } \mu(A) < \infty,\, \epsilon > 0 \}
	\]
	makes an open base for some topology on $\mathcal{M}(X, \mathcal{A}, \mu)$, which makes it a topological group. This topology, let us call the measure finite topology or the $F_\mu$-topology and designate $\mathcal{M}(X,\mathcal{A},\mu)$ equipped with this topology by the symbol $\mathcal{M}_F(X,\mathcal{A},\mu)$. It is easy to note that if $\mu(A) = 0$, then the basic open set $B(f, A, \epsilon)$ becomes the whole space $\mathcal{M}(X, \mathcal{A}, \mu)$ for any choice of $f$ and $\epsilon$. Consequently, if $\mu(X) = 0$, then $\mathcal{M}_F(X, \mathcal{A}, \mu)$ reduces to the indiscrete topological space. To avoid this undesirable situation, we assume that $\mu(X) >0$, we make one further assumption viz., that every measurable set (i.e., a member of $\mathcal{A}$) with infinite $\mu$-measure contains a measurable set with strictly positive finite measure. This second hypothesis is vacuously satisfied in a finite measure space i.e., for which $0<\mu(X)<\infty$. In this article, a few natural interplay existing between the topological structure of $\mathcal{M}_F(X, \mathcal{A}, \mu)$ and the measure $\mu$ on the measurable space $(X,\mathcal{A})$, comes into the surface. Thus, for instance, $\mathcal{M}_F(X, \mathcal{A}, \mu)$ becomes {path connected (or, connected)} if and only if $\mu$ is an {atomic measure} of a special type (Theorem \ref{equivpath}). If $\mu$ is a {nonatomic measure} (say the Lebesgue measure on the $\sigma$-algebra of all Lebesgue measurable subsets restricted on a measurable subset of $\mathbb{R}$), then $\mathcal{M}_F(X, \mathcal{A}, \mu)$ becomes first countable if and only if $\mu(X) < \infty$ (Theorem \ref{firstcounttable}).
	
	Now let us briefly narrate the organization of this article highlighting mainly on the principal results. An $f \in \mathcal{M}(X, \mathcal{A}, \mu)$ is called {essentially bounded} on a measurable set $G\in \mathcal{A}$, if $f$ is bounded on $G$ almost everywhere (with respect to the measure $\mu$). Suppose $L_G^\infty(X, \mathcal{A}, \mu)$ is the aggregate of all members of $\mathcal{M}(X, \mathcal{A}, \mu)$ which are essentially bounded on $G$.
	
	In Section 2, we prove amongst others that for any $G \in \mathcal{A}$, $L_G^\infty(G, \mathcal{A}, \mu)$ is a clopen set in the space $\mathcal{M}_F(X, \mathcal{A}, \mu)$. Using this fact, we compute in Section 3, the {component}, {the path component} and the {quasi component} of the function $\underline{0}$ in this space. It turns out that these three components are the same (Theorem 3.2). This eventually leads to the fact that $\mathcal{M}_F(X, \mathcal{A}, \mu)$ becomes {path connected} / {connected} if and only if it is a {topological ring}/{ topological vector space} (Theorem \ref{equivpath}). We would like to mention in this context that this last condition on $\mathcal{M}_F(X, \mathcal{A}, \mu)$ is equivalent to the $\aleph_0$-boundedness of the topological group $\mathcal{M}_F(X, \mathcal{A}, \mu)$. A topological group $(G,+)$ is  called $\aleph_0$-bounded if given any open neighborhood $U$ of $0$ in $G$, there exists a countable set $A\subseteq G$ such that $A+U=G$ \cite{KH}. Incidentally, we prove that if $\mathcal{M}_F(X, \mathcal{A}, \mu)$ is locally compact then it becomes {path connected} (Remark 3.12). However, for a large class of measures $\mu$, indeed for all {nonatomic }measures $\mu$ it is realized that $\mathcal{M}_F(X, \mathcal{A}, \mu)$ is nowhere locally compact.	Before stating the principal results in Section 4, we feel it imperative to recall the notions: {atomic} and {nonatomic measures}. A measurable set $A $ $\in \mathcal{A}$ is called an {atom} if $\mu(A) > 0$ and we cannot decompose $A$ in the manner:
	\[
	A = B \sqcup C, \quad B, C \in \mathcal{A},\mu(B) >0 \text{ and } \mu(C) >0.
	\]
	The measure $\mu$ is called an {atomic measure} on $\mathcal{A}$ if there exists at least one atom in the $\sigma$-algebra; otherwise, $\mu$ is called a {nonatomic measure}. An {atomic measure} $\mu$ on $\mathcal{A}$ is called {purely atomic} if each measurable set with strictly positive $\mu$-measure contains an atom. For further information on an {atomic }and {nonatomic} measures see \cite{10}.
	
	 In Section 4, we define the notion {hemifiniteness} for a measure $\mu$ in a manner somewhat analogous to the notion of {hemicompactness} in topological spaces and realize that if $\mu$ is a {nonatomic measure}, then $\mu$ is {hemifinite} if and only if $\mu$ is finite (Theorem 4.2). However, for any type of measure $\mu$, whether it is {atomic} or {nonatomic}, we prove that the {hemifiniteness} of $\mu$ is equivalent to the {first countability }of $\mathcal{M}_F(X, \mathcal{A}, \mu)$ (Theorem \ref{firstcounttable}). 
	 
	 In Section 5, we show that various cardinal functions associated with the space $\mathcal{M}_F(X, \mathcal{A}, \mu)$ and the $\sigma$-algebra $\mathcal{A}$ on $X$ become identical. Indeed, we prove that three cardinal functions, viz., the {tightness}, the {$\pi$-character} and the {character }of the topological space $\mathcal{M}_F(X, \mathcal{A}, \mu)$ become identical and are equal to the controlling number of the $\sigma$-algebra $\mathcal{A}$, as is defined appropriately (Theorem \ref{character}). Afterwards, we prove certain inequalities involving the {weight, density, cellularity} and {Lindel\"{o}f number} of $\mathcal{M}_F(X,\mathcal{A},\mu)$.
	
	\section{A few preliminary results}
	
	The $U_\mu$-topology and the $m_\mu$-topology on the ring $\mathcal{M}(X, \mathcal{A}, \mu)$ are introduced in \cite{1,2,3} as measure-theoretic counterparts of the $u$-topology and the $m$-topology on $C(X)$, respectively. We would like to recall that a typical basic open set in the space $\mathcal{M}(X, \mathcal{A}, \mu)$ equipped with $U_\mu$-topology is $B(f, X, \epsilon)$ with $\epsilon>0$ a real number, and a typical basic open set in $\mathcal{M}(X, \mathcal{A}, \mu)$ with $m_\mu$-topology is $B(f,X,\epsilon)$, where $\epsilon\in \mathcal{M}(X,\mathcal{A},\mu)$, $\epsilon(x)>0$ for all $x\in X$ and
	$$
	B(f, X,\epsilon) = \{ g \in \mathcal{M}(X,\mathcal{A},\mu) : \sup_{x \in X \setminus A_g} |f(x) - g(x)| < \epsilon\text{ for }A_g\in \mathcal{A}\text{ with }\mu(A_g)=0 \}
	$$
	It is easy to check that the $F_\mu$-topology is weaker than the $U_\mu$-topology on $\mathcal{M}(X, \mathcal{A}, \mu)$, which is further weaker than the $m_\mu$-topology on it. For convenience, let us denote by $\mathcal{M}_u(X, \mathcal{A}, \mu)$ ($\mathcal{M}_m(X, \mathcal{A}, \mu)$)  the set $\mathcal{M}(X, \mathcal{A}, \mu)$ with the $U_\mu$-topology (respectively the $m_\mu$-topology). If $\mu(X)<\infty$, then it is trivial that $\mathcal{M}_F(X, \mathcal{A}, \mu)=\mathcal{M}_u(X, \mathcal{A}, \mu)$. On the other hand if $\mu(X)=\infty$, then for any $A\in \mathcal{A}$ with $0<\mu(A)<\infty$ and $\delta>0$, it will never happen that $\underline{0}\in B(\underline{0}, A,\delta)\subseteq B(\underline{0}, X,1)$, because the function $f\in \mathcal{M}(X, \mathcal{A}, \mu)$ defined by $f(x)=\begin{cases}
		\frac{1}{2}\delta&\text{ if } x\in A\\
		2,& \text{ if }x\in X\setminus A
	\end{cases}$ belongs to $B(\underline{0}, A,\delta)$ but is outside the set $B(\underline{0},X,1)$. Consequently, $B(\underline{0},X,1)$ is not an open set in the space $\mathcal{M}_F(X, \mathcal{A}, \mu)$. Thus we already prove the following result.
\begin{theorem}
	$\mathcal{M}_F(X, \mathcal{A}, \mu)=\mathcal{M}_u(X, \mathcal{A}, \mu)$ if and only if $\mu(X)<\infty$.
\end{theorem}
 It follows from the  Corollary 3.24 in \cite{3} that $\mathcal{M}_u(X, \mathcal{A}, \mu)=\mathcal{M}_m(X, \mathcal{A}, \mu)$ if and only if each function in $\mathcal{M}(X, \mathcal{A}, \mu)$ is essentially bounded on $X$, i.e., $\mathcal{M}(X, \mathcal{A}, \mu)=L_X^\infty(X, \mathcal{A}, \mu)=L^\infty(X, \mathcal{A}, \mu)$. This yields to the following simple fact:
 \begin{theorem}
 	$\mathcal{M}_F(X, \mathcal{A}, \mu)=\mathcal{M}_m(X, \mathcal{A}, \mu)$ if and only if $\mu(X)<\infty$ and each function in $\mathcal{M}(X, \mathcal{A}, \mu)$ is  essentially bounded on $X$.\end{theorem} Now it follows from Lemma 7.5 in \cite{3} that if $f\in \mathcal{M}(X, \mathcal{A}, \mu)$ and $G$ is a $\mu$-atom, then $f\in L^\infty_G(X,\mathcal{A},\mu)$. A more sharp result follows from this fact as is manifested by the next theorem:
 	 
\begin{theorem}\label{atom}
Let $G$ be a $\mu$-atom and $f\in \mathcal{M}(X, \mathcal{A}, \mu)$. Then $f$ is constant almost everywhere on $G$.
\end{theorem}
\begin{proof}
	Without loss of generality, we can assume that $f\geq 0$ on $X$. If $f$ is a simple function, then this result is obtained almost immediately from the definition of a $\mu$-atom, on applying Lemma 7.5 in \cite{3} as reproduced above. Therefore, assume that $f$ is not a simple function. Then from a standard result in measure theory \cite[Theorem 2.89]{8} there exists a monotonic increasing sequence of simple functions $$f_1(x)\leq f_2(x)\leq f_3(x)\leq \cdots\text{ for all } x\in X \text{ such that} \lim\limits_{n\rightarrow\infty}f_n(x)=f(x)$$ holds over $X$. Since each simple function $f_n$ is constant almost everywhere on $G$, as observed above, there exists a measurable subset $G_0$ of $G$ with $\mu(G_0)=0$ and a sequence $\{c_n\}_n$ of non negative real numbers such that $f_n(x)=c_n$ for all $n\in \mathbb{N}$ and for all $x\in G\setminus G_0$. It follows that$f(x)=\lim\limits_{n\rightarrow\infty}c_n$ for all $x\in G\setminus G_0$.
\end{proof}
Let $\mathcal{A}_>=\{A\in\mathcal{A}:0<\mu(A)<\infty\}$. Then the theorem that follows will be useful to us in the next section. \begin{theorem}\label{clopen}
	 \begin{enumerate}
	 	\item For any $G\in \mathcal{A}_>$, $L_G^\infty$ is a clopen set in $ \mathcal{M}_F(X, \mathcal{A}, \mu)$.
	 	\item For any $G\in \mathcal{A}$ with $\mu(G)>0$, $L_G^\infty(X,\mathcal{A},\mu)=\mathcal{M}(X,\mathcal{A},\mu)$ if and only if $G$ can be written as the union of a finite number of $\mu$-atoms.
	 \end{enumerate}
\end{theorem}
\begin{proof}
	1. For any $f\in L^\infty_G(X,\mathcal{A},\mu)$, $B(f,G,1)\subseteq L^\infty_G(X,\mathcal{A},\mu)$, this proves that $L^\infty_G(X,\mathcal{A},\mu)$ is open in $\mathcal{M}_F(X,\mathcal{A},\mu)$. On the other hand if $g\notin L^\infty_G(X,\mathcal{A},\mu)$, then $B(g,G,1)\cap L^\infty_G(X,\mathcal{A},\mu)=\emptyset$. Hence, $L^\infty_G(X,\mathcal{A},\mu)$ is a closed subset of $\mathcal{M}_F(X,\mathcal{A},\mu)$.
	
	2. If $G$ is the union of a finite number of $\mu$-atoms, then it follows in a straightforward manner from Theorem \ref{atom} that $\mathcal{M}(X,\mathcal{A},\mu)=L^\infty_G(X,\mathcal{A},\mu)$. To prove the other part of this result, assume that $L^\infty_G(X,\mathcal{A},\mu)=\mathcal{M}(X,\mathcal{A},\mu)$. If possible, let $G$ cannot be written as the union of finitely many $\mu$-{atoms} $\cdots$(*).
	
	 There are two cases:

	Case-I: There is a $B\in \mathcal{A}$ with $\mu(B)>0$ such that $B\subseteq G  $ and $B$ does not contain a $\mu$-{atom}, in particular $G$ itself is not a $\mu$-{atom}. So, $G$ has a decomposition, $G=B_1\cup B_2$, where $\mu(B_1)>0$ and $\mu(B_2)>0$. Since none of $B_1$ and $B_2$ is an {atom}, each of $B_1$ and $B_2$ has a decomposition into two sets belonging to $\mathcal{A}$ with positive measure. An obvious induction therefore yields a pairwise disjoint sequence $\{B_n\}_n$ of members of $\mathcal{A}$ such that each $B_n\subseteq G$ and $\mu(B_n)>0$ for all $n\in \mathbb{N}$.

	Case-II: Each member of $\mathcal{A}$ with $\mu$-measure $>0$, contained in $G$ does contain an {atom}. Now, since by the assumption (*), $G$ is not a $\mu$-{{atom}}, $G$ has a decomposition, $G=B_1\cup B_2$, where $B_1, B_2\in \mathcal{A}$ with $\mu(B_1)>0$, $\mu(B_2)>0$. From the same assumption (*), it follows that at least one of $B_1$ and $B_2$ is a non $\mu$-{atom}, say $B_1$ is not a $\mu$-{atom}. Then $B_1$ has a decomposition $B_1=B_{11}\cup B_{12},$ where $B_{11},B_{12}\in \mathcal{A}$ with $\mu(B_{11})>0$, $\mu(B_{12})>0$ and from (*) we get at least one of $B_{11}$, $B_{12}$ and $B_2$ is a non $\mu$-{atom}. An obvious induction therefore results in a pairwise disjoint sequence of members $\{B_n\}_n$ of $\mathcal{A}$, each $B_n$ begin contained in $G$ and $\mu(B_n)>0$ for all $n\in \mathbb{N}$. 
	Thus in any case, we have a pairwise disjoint sequence $\{B_n\}_n$ of members of $\mathcal{A}$ with $\mu(B_n)>0$ such that $B_n\subseteq G$ for each $n\in \mathbb{N}$. Construct the function $f\in \mathcal{M}(X,\mathcal{A},\mu)$ as follows: $$f(x)=\begin{cases}
		n,\text{ when }x\in B_n,\text{ for all }n\in \mathbb{N}\\
		0,\text{ elsewhere }
	\end{cases}$$ Surely, $f\notin L^\infty_G(X,\mathcal{A},\mu)$. We arrive at a contradiction.
\end{proof}
\begin{corollary}
	The $u_\mu$-topology= $m_\mu$-topology on $\mathcal{M}(X,\mathcal{A},\mu)$ if and only if $X$ can be written as the union of finitely many {atoms}. This follows from Corollary 3.2.4 in \cite{3} and Theorem 2.4 (2). Thus for an infinite {nonatomic measure} $\mu$ on the $\sigma$-algebra $\mathcal{A}$ on $X$, the $F_\mu$-topology $\subsetneq$ the $u_\mu$-topology $\subsetneq$ the $m_\mu$-topology.
\end{corollary}
\section{ Connectedness and local compactness of $\mathcal{M}_F(X,\mathcal{A},\mu)$}
We start with the following lemma which helps us to compute the component of $\underline{0}$ in the $F_\mu$-topology.\begin{lemma}\label{pathconnected}
	$\bigcap\limits_{G\in \mathcal{A}_>}L^\infty_G(X,\mathcal{A},\mu)$ is {path connected} in the $F_\mu$-topology, appropriately relativized.
\end{lemma}
\begin{proof}
	Let $f\in \bigcap\limits_{G\in \mathcal{A}_>}L^\infty_G(X,\mathcal{A},\mu)$. Since $\mathcal{M}_F(X,\mathcal{A},\mu)$ is a homogeneous space, it is enough to show that $f$ is pathwise connected to $\underline{0}$. For that purpose, define the map $p:[0,1]\rightarrow \mathcal{M}_F(X,\mathcal{A},\mu)$ by the rule $p(t)=t\cdot f$. It is clear that $p([0,1])\subseteq \bigcap\limits_{G\in \mathcal{A}_>}L^\infty_G(X,\mathcal{A},\mu)$ and $p(0)=\underline{0}$, $p(1)=f$. Towards proving the continuity of $p$ at a point $t\in [0,1]$, let $B(p(t),F,\epsilon)$ be any basic open set in $\mathcal{M}_F(X,\mathcal{A},\mu)$ containing $p(t)$, here $F\in \mathcal{A}_>$ and $\epsilon>0$. Now since $f\in L_F^\infty(X,\mathcal{A},\mu)$, there is an $M>0$ in $\mathbb{R}$ such that $|f(x)|<M$ for all $x$, almost everywhere $(\mu)$ on $F$. Set $\delta=\frac{1}{4M}\epsilon$, then for any $t'\in [0,1]\cap(t-\delta, t+\delta)$, $|p(t')(x)-p(t)(x)|<\epsilon$ for all x, almost everywhere $(\mu)$ on $F$. This settles the continuity of $p$ at the point $t$.
\end{proof}
\begin{theorem}\label{component}
	The component $C$, the quasi component $Q$ and the path component $P$ of the function $\underline{0}$ is the space $\mathcal{M}_F(X,\mathcal{A},\mu)$ coincide and are equal to $\bigcap\limits_{G\in \mathcal{A}_>}L^\infty_G(X,\mathcal{A},\mu)$. 
\end{theorem}
\begin{proof}
	 It is well known that $P\subseteq C\subseteq Q$. It follows from Lemma \ref{pathconnected} that $\bigcap\limits_{G\in \mathcal{A}_>}L^\infty_G(X,\mathcal{A},\mu)$$\subseteq P$. On the other hand, Theorem \ref{clopen}(1) implies that $Q\subseteq \bigcap\limits_{G\in \mathcal{A}_>}L^\infty_G(X,\mathcal{A},\mu)$. Thus, $P=C=Q=\bigcap\limits_{G\in \mathcal{A}_>}L^\infty_G(X,\mathcal{A},\mu)$.
\end{proof}
\begin{corollary}
	The space $\mathcal{M}_F(X,\mathcal{A},\mu)$ is never totally disconnected because \newline $L^\infty(X,\mathcal{A},\mu)\equiv L_X^\infty(X,\mathcal{A},\mu)\subseteq  \bigcap\limits_{G\in \mathcal{A}_>}L^\infty_G(X,\mathcal{A},\mu)$ and each real valued constant function $\underline{r}$, $r\in \mathbb{R}$ belongs to $L^\infty(X,\mathcal{A},\mu)$.
\end{corollary}
The next theorem gives a measure theoretic characterization of path connectedness of $\mathcal{M}_F(X,\mathcal{A},\mu)$.
\begin{theorem}\label{equivpath}
	The following statements are equivalent-
	\begin{enumerate}
		\item $\mathcal{M}_F(X,\mathcal{A},\mu)$ is path connected.
		\item $\mathcal{M}_F(X,\mathcal{A},\mu)$ is connected.
		\item $\mathcal{M}_F(X,\mathcal{A},\mu)=\bigcap\limits_{G\in \mathcal{A}_>}L^\infty_G(X,\mathcal{A},\mu)$.
		\item Any measurable set with finite non zero $\mu$-measure can be written as the union of a finite number if $\mu$-{atom}s.
		\item $\mathcal{M}_F(X,\mathcal{A},\mu)$ is a topological ring.
		\item $\mathcal{M}_F(X,\mathcal{A},\mu)$ is a topological vector space.
		\item $\mathcal{S}\equiv$ the collection of all simple functions, is dense in $\mathcal{M}_F(X,\mathcal{A},\mu)$.
		\item $\mathcal{I}\equiv$ the collection of all integrable functions, is dense in $\mathcal{M}_F(X,\mathcal{A},\mu)$.
		\item $\mathcal{M}_F(X,\mathcal{A},\mu)$ is {$\aleph_0$-bounded}.
	\end{enumerate}
\end{theorem}
\begin{proof}
	$(1)\implies (2)$ is trivial. $(2)\implies (3)$ follows from Theorem \ref{component}. $(3)\implies (1)$ follows from Lemma \ref{pathconnected}. $(3)\iff (4)$ follows from Theorem \ref{clopen} (2). 
	
	$(3)\implies (5)$: Let $f,g\in \mathcal{M}(X,\mathcal{A},\mu)$ and $B(fg,F,\epsilon)$ be a basic open set containing $fg$ in the space $\mathcal{M}_F(X,\mathcal{A},\mu)$, here $F\in \mathcal{A}_>$ and $\epsilon>0$. Then there exists $M>0$ in $\mathbb{R}$ such that $|f(x)|\leq M$ and $|g(x)|\leq M$ for all $x\in F$, almost everywhere($\mu$) on it. Choose any $\delta>0$ such that $\delta^2+2M\delta<\epsilon$. Thus $B(f,F,\delta)B(g,F,\delta)\subseteq B(fg,F,\epsilon)$. Thus it is proved that $\mathcal{M}_F(X,\mathcal{A},\mu)$ is a topological ring.
	
	$(5)\implies (3)$: Let $f\in \mathcal{M}(X,\mathcal{A},\mu)$ and $G\in \mathcal{A}_>$. Now $B(\underline{0},G,1)$ is a basic open set in $\mathcal{M}_F(X,\mathcal{A},\mu)$ containing $\underline{0}$. By the assumed condition $(5)$, there exists $\delta>0$ and $F_1,F_2\in \mathcal{A}_>$ such that $B(f,F_1,\delta)B(\underline{0},F_2,\delta)\subseteq B(0,G,1)$. It follows that $|f(x)\frac{\delta}{2}|<1$ for all $x$, almost everywhere($\mu$) on $G$ and hence $f\in L_G^\infty(X,\mathcal{A},\mu)$. 
	
	 An analogous proof can be constructed for the equivalence $(3)\iff (6)$.

	$(4)\implies (7)$: Let $f\in \mathcal{M}(X,\mathcal{A},\mu)$ and $B(f,F,\epsilon)$ be a basic open set in $\mathcal{M}_F(X,\mathcal{A},\mu)$ containing $f$, where $F\in \mathcal{A}_>$ and $\epsilon>0$. By (4), $F=F_1\cup F_2\cup\cdots \cup F_m$, where $F_i$'s are {atoms} for all $i=1,2,\cdots, m$. Without loss of generality, we can assume that $F_i\cap F_j=\emptyset$ where $i,j\in \{1,2,\cdots, m\}$ and $i\neq j$. By Theorem \ref{atom}, $f(x)=c_i$ for all $x\in F_i\setminus B_i$, where $\mu(B_i)=0$, $c_i\in \mathbb{R}$, for all $i=1,2,\cdots, m$. Now consider the simple function $s=\sum\limits_{i=1}^m c_i \chi_{F_i}$. Note that, $s\in \mathcal{S}\cap B(f,F,\epsilon)$. Consequently, $\mathcal{S}$ is dense in $\mathcal{M}_F(X,\mathcal{A},\mu)$.
	
	$(7)\implies (8)$: Let $f\in \mathcal{M}(X,\mathcal{A},\mu)$ and $B(f,F,\epsilon)$ be a basic open set containing $f$, where $F\in \mathcal{A}_>$ and $\epsilon>0$. By (7), we can find a simple function $s=\sum\limits_{i=1}^m c_i\chi_{A_i}\in B(f,F,\epsilon)$ which implies that $ \sum\limits_{i=1}^m c_i\chi_{A_i\cap F}\in B(f,F,\epsilon)$. Consider the function $$h(x)=\begin{cases}
		\sum\limits_{i=1}^m(c_i\chi_{A_i})(x),& \text{ when }x\in F\\
		0,&\text{ otherwise}
	\end{cases}$$ Then $h\in \mathcal{I}\cap B(f,F,\epsilon)$.

$(8)\implies (4)$: If possible, let there be a measurable set $F$ with strictly positive finite measure such that $F$ cannot be written as the union of a finite number of {atoms}. As in the proof of Theorem \ref{clopen}(2), We can find a sequence $\{A_n\}_n$ of disjoint measurable sets satisfying $A_n\subseteq F$ and $0<\mu(A_n)<\mu(F)$. Consider the measurable function $$h(x)=\begin{cases}
	\frac{n^2}{\mu(A_n)}&\text{, when }x\in A_n, \text{ }n\in \mathbb{N}\\
	0&\text{, otherwise}
\end{cases}$$ It is easy to check that $\int\limits_X h d\mu=\infty$ and hence $B(h,F,1)\cap \mathcal{I}=\emptyset$. Which is a contradiction.

$(4)\implies (9)$: Let $B(\underline{0},F,\epsilon)$ be a basic open set containing $\underline{0}$, where $F\in \mathcal{A}_>$ and $\epsilon>0$. By (4), $F=F_1\cup F_2\cup\cdots\cup F_m$, where $F_i$'s are {atoms} for all $i=1,2,\cdots ,m$. Without loss of generality, we can assume that $F_i\cap F_j=\emptyset$, where $i,j\in \{1,2,\cdots, m\}$ and $i\neq j$. Let $\mathbb{Q}=\{e_1,e_2,\cdots\}$ be an enumeration of $\mathbb{Q}$, the set of all rational numbers. Let $T=\{e_i \chi F_p:i\in \mathbb{N} \text{ and }p\in \{1,2,\cdots, m\}\}.$ Let $S=$ the set of all possible finite sums of members of $T$. Then $S$ is countable. It is easy to check, $\mathcal{M}(X,\mathcal{A},\mu)=B(\underline{0},F,\epsilon)+S$. Consequently, $\mathcal{M}_F(X,\mathcal{A},\mu)$ is $\aleph_0$-bounded.

$(9)\implies (4)$: If possible, let there be a measurable set $A$ with strictly positive finite measure such that $A$ cannot be written as the union of finite number of {atoms}. As in the proof of Theorem \ref{clopen}(2), we can find a sequence $\{A_n\}_n$ of disjoint measurable sets satisfying $A_n\subseteq F$ and $0<\mu(A_n)<\mu(A)$. Now, $B(\underline{0},A,1)$ is an open set containing $\underline{0}$. Since $\mathcal{M}_F(X,\mathcal{A},\mu)$ is {$\aleph_0$-bounded}, then there is a countable set $\{f_1,f_2,\cdots\}$ of measurable functions such that $B(\underline{0},A,1)+\{f_1,f_2,\cdots\}=\mathcal{M}(X,\mathcal{A},\mu)$. Consider the measurable function $$p(x)=\begin{cases}
	f_n(x)+2&\text{, when }x\in A_n,\text{ }n\in \mathbb{N}\\
	0&\text{, otherwise}
\end{cases}$$ Note that, $p\notin B(\underline{0},F,1)+\{f_1,f_2,\cdots\}$. Which is a contradiction.
\end{proof}
	\begin{corollary}
		A measure $\mu$ on the $\sigma$-algebra $\mathcal{A}$ on $X$ which renders the space $\mathcal{M}_F(X,\mathcal{A},\mu)$ path connected or connected is purely atomic. 
	\end{corollary}
The following is an example of a {purely atomic measure} $\mu$ on $G$ suitable measurable space $(X,\mathcal{A})$ for which $\mathcal{M}_F(X,\mathcal{A},\mu)$ is not path connected.

\begin{example}\label{example}
	Let $X=\mathbb{N}$, $\mathcal{A}=\mathcal{P}(\mathbb{N})=$ the power set of $\mathbb{N}$ and $\mu:\mathcal{P}(\mathbb{N})\rightarrow [0,\infty)$ be the measure defined as follows: $\mu(\{n\})=\frac{1}{2^n}$ for all $n\in \mathbb{N}$. Then $\mu$ is a {purely atomic measure} on $(X,\mathcal{A})$ and each singleton $\{n\}$, $n\in \mathbb{N}$ is an {atom} and these are the only {atoms} in this measure space. However $\mu(\mathbb{N})=\sum\limits_{n\in \mathbb{N}}\frac{1}{2^n}=1$, which shows that $\mathbb{N}$ is a $\mu$-measurable set with finite nonzero $\mu$-measure, which is never expressible as the union of finitely many $\mu$-{atoms}. It follows from Theorem \ref{equivpath} (vide equivalence of (1) and (4)) that $\mathcal{M}_F(X,\mathcal{A},\mu)$ is not path connected. 
\end{example} 
We are now going to prove that the space $\mathcal{M}_F(X,\mathcal{A},\mu)$ is never extremally disconnected. Recall that a topological space $Z$ is called extremally disconnected if the closure of every open set is open in $Z$. The following subsidiary result will be helpful towards proving such an assertion. For any $G\in \mathcal{A}_>$, let $G^+=\{f\in \mathcal{M}(X,\mathcal{A},\mu):f(x)>0 \text{ almost}$
$ \text{everywhere}(\mu)\text{ on }G\}$ and  $G^0=\{f\in \mathcal{M}(X,\mathcal{A},\mu):f(x)\geq0 \text{ almost}$  $\text{everywhere}(\mu) \text{ on }G\}$. \begin{lemma}\label{edlemma}
	\begin{enumerate}
		\item If $G\in \mathcal{A}_>$ is a $\mu$-{atom}, then $G^+$ is an open set in $\mathcal{M}_F(X,\mathcal{A},\mu)$. On the other hand if $\mu$ is non {atomic}, then for any $G\in \mathcal{A}_>$, $G^+$ is not open in $\mathcal{M}_F(X,\mathcal{A},\mu)$.
		\item  If $G$ is not a $\mu$-{atom}, then the closure of the interior of $G^+$ in $\mathcal{M}_F(X,\mathcal{A},\mu)$ is $G^+$ itself.
		\item If $G$ is a $\mu$-{atom}, then the closure of $G^+$ in $\mathcal{M}_F(X,\mathcal{A},\mu)$ is $G^0$ and $G^0$ is not open in $\mathcal{M}_F(X,\mathcal{A},\mu)$.
	\end{enumerate}
\end{lemma}
\begin{proof}
	(1). First let $G$ be a $\mu$-{atom} and $f\in G^+$. Then it follows from Theorem \ref{atom} that, there exists a $c>0$ in $\mathbb{R}$ such that $f(x)=c$ almost everywhere($\mu$) on $G$. Consequently, $B(f,G,\frac{1}{2}c)\subseteq G^+$, thus $f$ is an interior point of $G^+$ and hence $G^+$ is open in $\mathcal{M}_F(X,\mathcal{A},\mu)$.
	
	Let $\mu$ be {non atomic} and $G\in \mathcal{A}_>$, then $G$ is not a $\mu$-{atom}. It follows from Lemma 7.1 in \cite{3} that $G$ has a decomposition: $G=\bigcup\limits_{n\in \mathbb{N}} G_n$, where $G_n\cap G_m=\emptyset$ if $m\neq n$ and $\mu(G_n)>0  $ for each $n\in \mathbb{N}$. Let $f:X\rightarrow \mathbb{R}$ be defined as follows, $$f(x)=\begin{cases}
		\frac{1}{n},&\text{ when }x\in G_n,\text{ where }n\in \mathbb{N}\\
		0,&\text{ elsewhere }
		
	\end{cases}$$ Then $f\in G^+$ and $f$ has the following property: for all sufficiently small $\epsilon>0$ in $\mathbb{R}$, $f<\epsilon$ on a set $G$ of positive $\mu$-measure. We assert that $f$ is not an interior point of $G^+$ and hence $G^+$ is not open in $\mathcal{M}_F(X,\mathcal{A},\mu)$.

	\underline{ Proof of the assertion:} If possible, let $f$ be an interior point of $G^+$. Then there exist $H\in \mathcal{A}_>$ and $\epsilon>0$ in $\mathbb{R}$ such that $f\in B(f,H,\epsilon)\subseteq G^+$. It is clear that $f-\frac{1}{2}\epsilon \in B(f,H,\epsilon)$, hence $f-\frac{1}{2}\epsilon\in G^+$. But since $f<\frac{1}{2}\epsilon$ on a set of positive $\mu$-measure, this implies that $f-\frac{1}{2}\epsilon$ takes negative values on a subset of $G$ of a positive $\mu$-measure and hence $f-\frac{	1}{2}\epsilon\notin G^+$, a contradiction.
	
	(2): Let $\mu$ be non {atomic} and $G\in \mathcal{A}_>$. We have seen in (1) that $G^+$ contains at least one function $f$ with the property that for sufficiently small $\epsilon>0$, $f<\epsilon$ on a set in $G$ of positive $\mu$-measure and such function $f$ does not belong to the interior of $G^+$.
	
	 Let $G^*=\{g\in G^+:\text{ there exists an }\epsilon>0\text{ such that }g\geq \epsilon\text{ almost everywhere}(\mu)$  $\text{ on }G\}$. Then for $g\in G^*$, $B(g,G,\frac{1}{2}\epsilon)\subseteq G^+$ (easy check), this shows that $g$ is an interior point of $G^+$. We have already shown in the first paragraph of the present proof that no point in $G^+$, outside $G^*$ can be an interior point of $G^+$. Hence, we can say that the interior of $G^+$ is $G^*$. We now confirm that $G^*$ is dense in $G^+$, indeed for any $f\in G^+$, we have $f(x)>0$ almost everywhere($\mu$) on $G$ and for $H\in \mathcal{A}_>$, $\epsilon>0$, $B(f,H,\epsilon)\cap G^*\neq \emptyset$. Indeed $f+\frac{1}{2}\epsilon\in B(f,H,\epsilon)\cap G^*$, because $f+\frac{1}{2}\epsilon>\frac{1}{2}\epsilon$ almost everywhere($\mu$) on $G$. Thus it is proved that $G^+=$ the closure of the interior of $G$ in $\mathcal{M}_F(X,\mathcal{A},\mu)$.
	 
	 (3): Here our hypothesis is $G$ is a $\mu$-{atom}. Let $f\in \mathcal{M}(X,\mathcal{A},\mu)\setminus G^0$. It follows from Theorem \ref{atom} that there exists a $c<0$ such that $f=\underline{c}$ almost everywhere($\mu$) on $G$. Consequently, $B(f,G,\frac{1}{2}|c|)\cap G=\emptyset$. This proves that $G^0$ is closed in $\mathcal{M}_F(X,\mathcal{A},\mu)$. It is not hard to show that $\underline{0}\in G^0$ is not an interior point of $G^0$ in the space $\mathcal{M}_F(X,\mathcal{A},\mu)$ (and $G^0$ is therefore not an open set in $\mathcal{M}_F(X,\mathcal{A},\mu)$): we argue by contradiction and assume that $\underline{0}$ is an interior point of $G^0$, therefore there exists $c>0$ in  $\mathbb{R}$ and $H\in \mathcal{A}_>$ such that $B(\underline{0},H,\epsilon)\subseteq G^0$. But the function $-\frac{1}{2}\epsilon\in B(\underline{0},H,\epsilon)$, which implies that $-\frac{1}{2}\epsilon\in G^0$- a contradiction. To complete the theorem, it remains to show that $\cl(G^+)=G^0$ in $\mathcal{M}_F(X,\mathcal{A},\mu)$. Indeed if $g\in G^0$ and $ A\in \mathcal{A}_>$, $\epsilon>0$, then $B(g,A,\epsilon)\cap G^+\neq \emptyset$, because the function $g+\frac{1}{2}\epsilon$ is a member of this non empty set.
\end{proof}
\begin{theorem}
	For any type of measure $\mu$, atomic or non atomic, the topological space $\mathcal{M}_F(X,\mathcal{A},\mu)$ is never extremally disconnected. This follows immediately from Lemma \ref{edlemma}.
\end{theorem}
It is easy to see, as we prove by the following elementary arguments that $\mathcal{M}_F(X,\mathcal{A},\mu)$ is never compact. In fact, for any $G\in \mathcal{A}_>$, $\{B(f,G,1):f\in \mathcal{M}(X,\mathcal{A},\mu)\}$ is an open covering for this space and for any finite subfamily $\{B(f_i,G,1):i=1,3,\cdots, n\}$ of this cover one can easily check that $\bigcup\limits_{i=1}^nB(f_i,G,1)\subsetneq\mathcal{M}(X,\mathcal{A},\mu)$, indeed the function $f:X\rightarrow \mathbb{R}$ defined by $f(x)=(|f_1(x)|+2)(|f_2(x)|+2)\cdots(|f_n(x)|+2)$ is a measurable function on $X$, but $f\notin B(f_i,G,1)$ for any $i\in\{1,2,\cdots, n\}$.
We now show that for a large class of measures $\mu$, the space $\mathcal{M}_F(X,\mathcal{A},\mu)$ is not even locally compact.
\begin{theorem}\label{locallycompact}
	Let $(X,\mathcal{A},\mu)$ be a measure space such that, there exists an $F\in \mathcal{A}_>$ such that $F$ cannot be written as the union finitely many $\mu$-{atoms}. Then the interior of any compact subset $K$of $\mathcal{M}_F(X,\mathcal{A},\mu)$ is empty.
\end{theorem}
\begin{proof}
	If possible, let the interior of $K$ be non empty. Then there exists $f\in \mathcal{M}(X,\mathcal{A},\mu)$, $\epsilon>0$ and a $G\in\mathcal{A}_>$ such that $B(f,G,\epsilon)\subseteq K$ and this further implies that $B(f,F\cup G,\epsilon)\subseteq K$. Since $F$ cannot be written as the union of finitely many $\mu$-{atoms}, it follows from the proof of Theorem \ref{clopen} (2), that there is a sequence of pairwise disjoint members $\{B_n\}_n$ of $\mathcal{A}_>$ such that $B_n\subseteq F$ for each $n\in \mathbb{N}$. Since $K\subseteq \bigcup\limits_{g\in K}B(g,F\cup G,\frac{1}{4}\epsilon)$, the compactness of $K$ implies that there are finitely many $ g_1,g_2,\cdots ,g_n\in K$ such that $K\subseteq\bigcup\limits_{i=1}^n B(g_i,F\cup G,\frac{1}{4}\epsilon) $. Since for each $m\in \mathbb{N}$, $f+\frac{\epsilon}{2}\chi_{B_m}\in B(f,G,\epsilon)$, it follows that $f+\frac{\epsilon}{2}\chi_{B_m}\in \bigcup\limits_{i=1}^nB(g_i,F\cup G,\frac{\epsilon}{4})$. So, there exist $i,j,k\in \mathbb{N}$, $k\leq n$ and $i\neq j$ such that $\{f+\frac{\epsilon}{2}\chi_{B_i}, f+\frac{\epsilon}{2}\chi_{B_j}\}\subseteq B(g_k,F\cup G,\frac{\epsilon}{4})$. Consequently, there exist measurable sets $A_i,A_j\in \mathcal{A}$ with $\mu(A_i)=\mu(A_j)=0$ such that $|(f+\frac{\epsilon}{2}\chi_{B_p})-g_k(x)|<\frac{\epsilon}{4}$ for all $x\in (F\cup G)\setminus A_p$, $p\in \{i,j\}$. Now choose a point $y\in (F\cup G)\setminus (A_i\cup A_j)$ such a choice is permissible as $\mu(A_i\cup A_j)=0$, while $\mu(F\cup G)>0$. We see that $\frac{\epsilon}{2}<|\frac{\epsilon}{2}\chi_{B_i}(y)-\frac{\epsilon}{2}\chi_{B_j}(y)|=|(f+\frac{\epsilon}{2}\chi_{B_i})(y)-(f+\frac{\epsilon}{2}\chi_{B_j})(y)|=|((f+\frac{\epsilon}{2}\chi_{B_i})(y)-g_k(y))-((f+\frac{\epsilon}{2}\chi_{B_j})(y)-g_k(y))|<\frac{\epsilon}{4}+\frac{\epsilon}{4}=\frac{\epsilon}{2}$- a contradiction.
\end{proof}
\begin{remark}
	If $\mu$ is a {nonatomic measure} on $(X,\mathcal{A})$, then the space $\mathcal{M}_F(X,\mathcal{A},\mu)$ is nowhere locally compact. Thus the only candidates for the measure $\mu$ to render $\mathcal{M}_F(X,\mathcal{A},\mu)$ locally compact are the {atomic measures}. However the measure $\mu$ in Example \ref{example} on $\mathcal{P}(\mathbb{N})$ is a {purely atomic measure}, although is still not locally compact, this follows from Theorem \ref{locallycompact} above.
\end{remark}
\begin{remark}
	The hypothesis in Theorem \ref{locallycompact} viz., that, there exists a set $F\in \mathcal{A}_>$ which cannot be written as the union of finitely many atoms is sufficient only but not necessary for the validity of Theorem \ref{locallycompact}. Indeed, if $x\in X$ and $\delta_x$ is the Dirac delta measure on $\mathcal{P}(X)$ defined in the usual manner: $\delta_x(E)=\begin{cases}
		1\text{, if }x\in E\\
		0\text{, if }x\notin E
	\end{cases}$ $E\in \mathcal{P}(X)$, then $\mathcal{M}(X,\mathcal{P}(X),\delta_x)=\mathbb{R}^X$, which is locally compact if and only if $X$ is a finite set.
\end{remark}
\begin{remark}
	If $\mathcal{M}(X,\mathcal{A},\mu)$ is locally compact, then each $A\in \mathcal{A}_>$ is expressible as the union of finitely many $\mu$-atoms (by Theorem \ref{locallycompact}). Hence, we get from Theorem \ref{clopen}(2) and Lemma \ref{pathconnected} that $\mathcal{M}(X,\mathcal{A},\mu)$ is path connected/ connected in this case.

\end{remark}
\section{ Countability properties of $\mathcal{M}_F(X,\mathcal{A},\mu)$}
\begin{definition}
	We call a measure $\mu$ on the measurable space $(X,\mathcal{A})$ {hemifinite} if there exists a sequence $\{A_n\}_n$ of members of $\mathcal{A}_>$ with the following property: given $F\in \mathcal{A}_>$, there exists an $n\in \mathbb{N}$ such that $F$ is contained in $A_n$ almost everywhere($\mu$) i.e., $\mu(F\setminus A_n)=0$. Since at the very beginning in the introductory section, we have already made the hypothesis that every $\mu$-measurable set with infinite $\mu$-measure contains a measurable set with strictly positive finite $\mu$-measure, it follows from the above definition that $\mu(X\setminus \bigcup\limits_{i=1}^\infty A_n)=0$. In other words $(X,\mathcal{A},\mu)$ becomes a {$\sigma$-finite} measure space. Thus we have proved that every {hemifinite} measure space  is {$\sigma$-finite}. It is trivial that a finite measure space is {hemifinite}. Thus the condition hemifiniteness for a measure lies between finiteness and {$\sigma$-finiteness}. 
\end{definition}

 The following result tells us that for {nonatomic measure} spaces, {hemifiniteness} and and finiteness conditions coincide.
 \begin{theorem}
 	Let $(X,\mathcal{A},\mu)$ be a nonatomic {hemifinite} measure space. Then $\mu(X)<\infty$.
 \end{theorem}
\begin{proof}
	If possible, let $\mu(X)=\infty$. Since $\mu$ is {hemifinite}, there exists an increasing sequence $A_1\subseteq A_2\subseteq \cdots$ of members of $\mathcal{A}_>$ with the following property: given $G\in \mathcal{A}_>$, there exists $n\in \mathbb{N}$ such that $\mu(G\setminus A_n)=0$. Since $\mu(X\setminus \bigcup\limits_{n=1}^\infty A_n )=0$, without loss of generality, we can assume that $X=\bigcup\limits_{n=1}^\infty A_n$ and consequently, $X=\bigcup\limits_{n=1}^\infty(A_{n+1}\setminus A_n)$. But since, $\mu(X)=\infty$, this forces that $\mu(A_{n+1}\setminus A_n)>0$ for infinitely many n's and without loss of generality, we can assume that $\mu(A_{n+1}\setminus A_n)>0 $ for each $n\in \mathbb{N}$. Consider any $n\in \mathbb{N}$, then since $\mu$ is a {nonatomic measure}, we can write $A_{n+1}\setminus A_n=\bigcup\limits_{k=1}^\infty B_k$, where $\{B_k\}_k$ is a pairwise disjoint sequence of members of $\mathcal{A}_>$. By virtue of the countable additivity of the measure $\mu$, we can write $\mu(A_{n+1}\setminus A_n)=\sum\limits_{n=1}^\infty\mu(B_k)<\infty$ and hence $\lim\limits_{k\rightarrow\infty}\mu(B_k)=0$. Thus we can produce a $B_{k_n}$ for which $0<\mu(B_{k_n})<\frac{1}{2^n} $ and this we can do for each $n\in \mathbb{N}$. Let $B=\bigcup\limits_{n=1}^\infty B_{k_n}$. Then $\mu(B)=\sum\limits_{n=1}^\infty \mu(B_{k_n})\leq\sum\limits_{n=1}^\infty\frac{1}{2^n}=1$. So $B\in \mathcal{A}_>$. Therefore, there exists $n\in \mathbb{N}$ such that $\mu(B\setminus A_n)=0$. This implies that $\mu(B_{k_n}\setminus A_n)=0$. Since $B_{k_n}\subseteq A_{n+1}\setminus A_n$, as constructed in the earlier step, this further implies that $\mu(B_{k_n})=0$- a contradiction to the choice $\mu(B_{k_n})>0$ for each $n\in \mathbb{N}$. This theorem is completely proved.
\end{proof}
\begin{remark}
	The conclusion of the above theorem may not be true for an {atomic measure}.
\end{remark}
\begin{example}[Example of an infinite {hemifinite} {atomic measure} space]
Let $X=\mathbb{N}$, $\mathcal{A}=\mathcal{P}(\mathbb{N})$ and $\mu:\mathcal{A}\rightarrow[0,\infty]$ be the counting measure defined in the usual manner. Then $\mu(\mathbb{N})=\infty$ and $\mathcal{B}=\{A\subseteq \mathbb{N}:\emptyset\neq A\text{ is a finite set}\}$ is a countable family of $\mu$-measurable sets with strictly positive finite $\mu$-measure. Furthermore, if $G\in \mathcal{A}$ such that $0<\mu(G)<\infty$, then $G=A$ for some $A\in \mathcal{B}$. This proves that $\mu$ is a {hemifinite} infinite measure. It is easy to check that $\mu$ is an {atomic measure}, because each $\{n\}$, $n\in \mathbb{N}$ is an atom- indeed $\mu$ is a purely {atomic measure}.
\end{example}
The following theorem contains several equivalent versions of first countability of the topological space $\mathcal{M}_F(X,\mathcal{A},\mu)$. For that purpose we need the following definitions.  A subset S of a space X is said to have{ countable character} if there exists a
sequence $\{W_n:n\in \mathbb{N}\}$ of open subsets in $X$ such that $S \subseteq W_n$ for all $n$ and if $W$ is
any open set containing $S$, then $W_n \subseteq W$ for some $n\in \mathbb{N}$ \cite{12}. A topological space $X$ is called { pointwise countable type }if each point is contained in a compact set with { countable character }\cite{12}. We also consider the more general property of being a { q-space}. This is a space such 
that for each point $a\in X$ there exists a sequence $\{U_n: n\in \mathbb{N}\}$ of neighborhoods of $a$ so that if $x_n\in U_n$, for each $n$ then $\{x_n: n\in \mathbb{N}\}$ has a cluster point in $X$ \cite{12}.
\begin{theorem}\label{firstcounttable}
	The following statements are equivalent for a measure space $(X,\mathcal{A},\mu)$:
	\begin{enumerate}
		\item $(X,\mathcal{A},\mu)$ is a {hemifinite} measure space.
		\item $\mathcal{M}_F(X,\mathcal{A},\mu)$ is {first countable}.
		\item $\mathcal{M}_F(X,\mathcal{A},\mu)$ is of {pointwise countable type}.
		\item $\mathcal{M}_F(X,\mathcal{A},\mu)$ is a {q-space}.
	\end{enumerate}
\end{theorem}
\begin{proof}
	Since a {first countable }space is of {pointwise countable type} and a space of {pointwise countable type} is a {q-space}, we need to prove only $(1)\implies(2)$ and $(4)\implies(1)$.
	
	$(1)\implies(2)$: The assumption (1) tells us that there is a sequence $\{A_n\}_n$ of members of $\mathcal{A}_>$ with the following property: for any $G\in \mathcal{A}_>$, $\mu(G\setminus A_n)=0$ for some $n\in \mathbb{N}$. Let $\mathcal{B}=\{B(\underline{0}, A_n,\frac{1}{m}:m,n\in \mathbb{N})\}$. Then $\mathcal{B}$ is a countable family of open neighborhoods of $\underline{0}$ in $\mathcal{M}_F(X,\mathcal{A},\mu)$. We claim that $\mathcal{B}$ is an open base about $\underline{0}$ in the same space and hence, $\mathcal{M}_F(X,\mathcal{A},\mu)$ is {first countable}. Proof of the claim: Let $G\in \mathcal{A}_>$, then $\mu(G\setminus A_n)=0$ for some $n\in \mathbb{N}$, it follows that $B(\underline{0},A_n,\frac{1}{m})\subseteq B(\underline{0},G,\frac{1}{m})$ for any $m\in \mathbb{N}$.
	
	$(4)\implies (1)$: Since $(4)$ is true, there exists a sequence of basic open sets $\{B(\underline{0},A_n,\epsilon_n):n\in \mathbb{N}\}$ about the point $\underline{0}$ in $\mathcal{M}_F(X,\mathcal{A},\mu)$ such that if $f_n\in B(\underline{0},A_n,\epsilon_n)$, $n\in \mathbb{N}$, then $\{f_n:n\in \mathbb{N}\}$ has a cluster point in this space. For each $n\in \mathbb{N}$, let $B_n=A_1\cup A_2\cup\cdots\cup A_n$. Then $\{B_n:n\in \mathbb{N}\}$ is an increasing sequence of sets in $\mathcal{A}_>$. Now choose $G\in \mathcal{A}_>$. We claim that there exists $n\in \mathbb{N}$ for which $\mu(G\setminus B_n)=0$ and hence $\mu$ becomes a {hemifinite measure} on $(X,\mathcal{A})$. If possible, let for each $n\in \mathbb{N}$, $\mu(G\setminus B_n)>0$. Let $g_n=n\chi_{G\setminus B_n}$. Then $g_n\in B(\underline{0},B_n,\frac{\epsilon_n}{2})\subseteq B(\underline{0}, A_n,\epsilon_n)$. From the assumption made above the sequence $\{g_n\}_n$ has a cluster point $g$ in the space $\mathcal{M}_F(X,\mathcal{A},\mu)$. So the open set $B(g,G,1)$ contains infinitely many $g_k$'s. Choose one such $g_k$ in this open set. Then $|g_k(x)-g(x)|<1$ almost everywhere($\mu$) on the set $G$ and hence $|k-g(x)|<1$ almost everywhere($\mu$) on the set $G\setminus B_k$ i.e., $k-1<g(x)<k+1$ for all $x\in (G\setminus B_k)\setminus A$ for some $A\in \mathcal{A}$ with $\mu(A)=0$$\cdots$ (i). We repeat this argument to find out an $m\in \mathbb{N}$ with $k+2<m$ such that $g_m\in B(g,G,1)$ and therefore, there exists $B\in \mathcal{A}$ with $\mu(B)=0$ for which we can write $m-1<g(x)<m+1$ for all $x\in (G\setminus B_m)\setminus B$ $\cdots$(ii). Since $\mu(A\cup B)=0$ and $\mu(G\setminus B_m)>0$ we can choose a point $c\in (G\setminus B_m)\setminus (A\cup B)$, for which we can write in view of the inequalities (i) and (ii): $k-1<g(c)<k+1<m-1<g(c)<m+1$, a contradiction (we have used the fact $k+2<m\implies B_k\subseteq B_m$, therefore $c\notin B_m\implies c\notin B_k$). The theorem in completely proved. 
\end{proof}
 We recall that a topological space is called an {$E_0$-space}(or a space of {countable pseudo character}) if each one-pointic set is a $G_\delta$-subset and every first countable $T_1$ topological space is {$E_0$}. We shall complete this section after giving a characterization of a prototype of $E_0$-\text{spaces} in the measure theoretic setting.
 \begin{definition}
 	For any $f\in \mathcal{M}(X,\mathcal{A},\mu) $, let $$[f]=\{g\in \mathcal{M}(X,\mathcal{A},\mu):f=g \text{ almost everywhere}(\mu)\text{ on }X\}$$ We call $\mathcal{M}_F(X,\mathcal{A},\mu)$ is a {$E_0^\mu$-space} if for each $f\in \mathcal{M}(X,\mathcal{A},\mu)$, $[f]$ is a $G_\delta$-set in $\mathcal{M}_F(X,\mathcal{A},\mu)$. It is easy to see that if $\mathcal{A}=\mathcal{P}(X)$ and $\mu$ is the counting measure on $\mathcal{A}$, then $\mathcal{M}_F(X,\mathcal{A},\mu)$ becomes an {$E_o^\mu$-space }if and only if it is an {$E_0$-space}.
 \end{definition}
\begin{theorem}\label{sigma}
	For a measure space $(X,\mathcal{A},\mu)$, $\mathcal{M}_F(X,\mathcal{A},\mu)$ is an {$E_0^\mu$-space} if and only if $X$ is a {$\sigma$-finite} measure space.
\end{theorem}
\begin{proof}
First assume that $\mathcal{M}_F(X,\mathcal{A},\mu)$ is an {$E_0^\mu$-space}. Then we can write $[\underline{0}]=\bigcap\limits_{n=1}^\infty B(\underline{0},A_n,\epsilon_n)$ for a countable subfamily $\{A_n\}_n$ of $\mathcal{A}_>$ and for a suitable $\epsilon_n>0$ for each $n\in \mathbb{N}$. We assert that $\mu(X\setminus \bigcup\limits_{n=1}^\infty A_n)=0$ and hence $\mu$ becomes a {$\sigma$-finite} measure on $\mathcal{A}$ because $0<\mu(A_n)<\infty$ for each $n\in \mathbb{N}$. Towards the proof of the last assertion, we argue by contradiction and assume that $\mu(X\setminus \bigcup\limits_{n=1}^\infty A_n)>0$. Then $\chi_{X\setminus\bigcup\limits_{n=1}^\infty A_n}\in \bigcap\limits_{n=1}^\infty B(\underline{0},A_n,\epsilon_n)$. But it is clear that $\chi_{X\setminus\bigcup\limits_{n=1}^\infty A_n}\notin[\underline{0}]$. We arrived at a contradiction.

Conversely, let $\mu$ be a {$\sigma$-finite} measure. Then we can write $X=\bigcup\limits_{n=1}^\infty A_n$ where $A_1\subseteq A_2\subseteq\cdots$ is an increasing sequence of measurable sets with $\mu(A_n)<\infty$ for each $n\in \mathbb{N}$. Choose, $f\in \mathcal{M}(X,\mathcal{A},\mu)$. It is sufficient to show that $[f]=\bigcap\limits_{n=1}^\infty B(f,A_n,\frac{1}{n})$. So choose $g\in \bigcap\limits_{n=1}^\infty B(f,A_n,\frac{1}{n})$. Then for each $n\in \mathbb{N}$, there exists $n\in \mathcal{A}$ with $\mu(B_n)=0$ such that $|g(x)-f(x)|<\frac{1}{n}$ for each $x\in A_n\setminus B_n$. On writing $B=\bigcup\limits_{n=1}^\infty B_n$, we get that $\mu(B)=0 $ and $|g(x)-f(x)|<\frac{1}{n}$ for all $x\in A_n\setminus B$$\cdots$(1). Now choose, $x\in X\setminus B$. Then there exists $n\in \mathbb{N}$ such that $x\in A_n\setminus B$ and consequently $x\in A_{n+k}\setminus B$ for each $k\in \mathbb{N}$, because $A_n\subseteq A_{n+1}\subseteq\cdots,$ this implies in view of the last inequality (1) that $|g(x)-f(x)|<\frac{1}{n+k}$ for all $k\in \mathbb{N}$, consequently $g(x)=f(x)$. Since $\mu(B)=0$, this implies $g=f$ almost everywhere($\mu$) on $X$, in other words $g\in [f]$.\end{proof}
\begin{corollary}
	If $ \mathcal{M}_F(X,\mathcal{A},\mu)$ is {first countable}, then it is an {$E_0^\mu$-space}. This follows from Theorem \ref{firstcounttable}, Theorem \ref{sigma} and the fact that a {hemifinite} measure $\mu$ is {$\sigma$-finite}.
\end{corollary}
\section{ Cardinal functions of the space $\mathcal{M}_F(X,\mathcal{A},\mu)$}
We recall the definition of a number of standard cardinal functions, to be needed in the present article.
 For any space \( X \) and any point \( x \) in \( X \), the {character} of \( x \) in \( X \), denoted by 
$
\chi(X,x),
$
is defined by
\[
\chi(X,x) = \aleph_0 + \min\{|\mathscr{B}_x| : \mathscr{B}_x \text{ is an open base for } X \text{ at } x\}.
\]

The {character} \( \chi(X) \) of \( X \) is defined by
$
\chi(X) = \sup\{\chi(X,x) : x \in X\}
$. Clearly a space \( X \) is {first countable} if and only if \( \chi(X) = \aleph_0 \).

The {weight} of a space \( X \) is defined by
\[
w(X) = \aleph_0 + \min\{|\mathscr{B}| : \mathscr{B} \text{ is an open base for } X\}.
\]

A space \( X \) is {second countable} if  and only if \( w(X) = \aleph_0 \).

The {density} \( d(X) \) of a space \( X \) is defined by
\[
d(X) = \aleph_0 + \min\{|D| : D \text{ is a dense subset of } X\}.
\]

A space \( X \) is {separable} if and only if \( d(X) = \aleph_0 \).

The {Lindelöf number} \( L(X) \) of \( X \) is defined by
\[
L(X) = \aleph_0 + \min\{ \mathfrak{m} : \text{every open cover of } X \text{ has a subcover of cardinality } \leq \mathfrak{m} \}.
\]

A space is {Lindelöf} if and only if \( L(X) = \aleph_0 \).

For a space $X$, the cellularity of $X$, denoted by $c(X)$, is defined by $$c(X) =
\aleph_0 + \sup\{|\mathcal{U} | : \mathcal{U} \text{ is a family of pairwise disjoint nonempty open subsets of } X\}.$$ A
space $X$ has countable chain condition if and only if $c(X) =\aleph_0$. 

 Let $\mathcal{V}$ be a collection of non-empty open 
sets in $X$. Then $\mathcal{V}$ is a local $\pi$-base for $x$ if for each open neighborhood 
$R$ of $x$, one has $V\subseteq R$ for some $V\in \mathcal{V}$. If in addition one has $x\in V$ for all 
$V\in \mathcal{V}$, then $V$ is a local base for $x$. The {$\pi$-character }of $x$ in $X$ is denoted by $\pi\chi(x,X)$ and defined by 
$$\pi\chi(x, X)= \min\{|\mathcal{V}|: \mathcal{V}\text{ is a local }\pi\text{ base for }x\} ; $$
The {$\pi$-character } of $X$ is denoted by $\pi(X)$ and defined by $$\pi\chi(X)=\aleph_0+\sup\{\pi\chi(x,X):x\in X\}.$$
So, $X$ has a { countable $\pi$-character }if and only if $\pi\chi(X)=\aleph_0$.

The tightness of $x$ in $X$ is denoted by $t(x,X)$ and defined by $$t(x,X)=\min\{k:\text{ for all }Y\subseteq X\text{ with }x\in \cl Y,\text{ there is }A\subseteq Y\text{ with }|A|\leq k\text{ and }x\in \cl A\}$$ { Tightness }of $X$ is denoted by $t(X)$ and defined by $$t(X)=\aleph_0+\sup\{t(x,X):x\in X\}$$ A space $X$ is {countably tight }if and only if $t(X)=\aleph_0$.

Additionally we introduce a new cardinal number in this section.\begin{definition}
	A subset $\mathcal{B}$ of $\mathcal{A}_> $ is said to be {controlling} the finite measure if for any $A\in  \mathcal{A}_>,$ there is a $B\in \mathcal{B}$ such that $\mu(A\setminus B)=0$. We now define the {controlling number} $cn(\mathcal{A})$ of the $\sigma$-algebra $\mathcal{A}$ on $X$ by the following formula:
	$$cn(\mathcal{A})=\aleph_0+\min\{|\mathcal{B}|:\mathcal{B}\subseteq \mathcal{A}_>\text{ and }\mathcal{B} \text{ controls the finite measure}\}$$
\end{definition} 

It is clear that $cn(\mathcal{A})=\aleph_0$ if and only if the measure $\mu$ is {hemifinite}. We establish the following four results interlinking the above mentioned cardinal functions for the space $\mathcal{M}_F(X,\mathcal{A},\mu)$. 
\begin{theorem}\label{character}
	For a measure space $(X,\mathcal{A},\mu)$, $$t(\mathcal{M}_F(X,\mathcal{A},\mu))=\pi\chi(\mathcal{M}_F(X,\mathcal{A},\mu))=\chi(\mathcal{M}_F(X,\mathcal{A},\mu))=cn(\mathcal{A}).$$ 
\end{theorem}\begin{proof}
Like any topological group, $\pi\chi(\mathcal{M}_F(X,\mathcal{A},\mu))=\chi(\mathcal{M}_F(X,\mathcal{A},\mu))$\cite{9}. Also in any topological space $Y$, $t(Y)\leq \chi(Y)$ \cite[Theorem 2.1]{j}. Therefore, it is sufficient to prove that: $\chi(\mathcal{M}_F(X,\mathcal{A},\mu))\leq cn(\mathcal{A})$ and $cn(\mathcal{A})\leq t(\mathcal{M}_F(X,\mathcal{A},\mu))$.

\underline{Proof of $\chi(\mathcal{M}_F(X,\mathcal{A},\mu))\leq cn(\mathcal{A}$)}: Let $\mathcal{F}\subseteq \mathcal{A}_>$ be a controlling set, controlling the finite measure. It is sufficient to show that $\{B(\underline{0},F,\frac{1}{n}):F\in \mathcal{F}, n\in \mathbb{N}\}$ is a local base at $\underline{0}$ in $\mathcal{M}_F(X,\mathcal{A},\mu)$. Let $G\in \mathcal{A}_>$ and $n\in \mathbb{N}$ be arbitrary. Then there exists $F\in \mathcal{A}_>$ such that $\mu(G\setminus F)=0$. Hence, $B(\underline{0},F,\frac{1}{n})\subseteq B(\underline{0},G,\frac{1}{n})$.

\underline{Proof of $cn(\mathcal{A})\leq t(\mathcal{M}_F(X,\mathcal{A},\mu))$}: Let $L=\{\chi_{X\setminus A}:A\in \mathcal{A}_>\}$. Then we assert that $\underline{0}\in cl L$ in the space $\mathcal{M}_F(X,\mathcal{A},\mu)$ because for $G\in \mathcal{A}_>$,$\epsilon>0$, $B(\underline{0},G,\epsilon)\cap L$ contains the function $\chi_{X\setminus G}$. Let $L'$ be a subset of $L$ with the smallest cardinal no with the property that $\underline{0}\in \cl L'$. Then $|L'|\leq t(\underline{0},\mathcal{M}_F(X,\mathcal{A},\mu))=$ the tightness of $\underline{0}$ in the space $\mathcal{M}_F(X,\mathcal{A},\mu)=t(\mathcal{M}_F(X,\mathcal{A},\mu))$ because, $\mathcal{M}_F(X,\mathcal{A},\mu)$ is a homogeneous space. It suffices to show that $cn(\mathcal{A})\leq t(\mathcal{M}_F(X,\mathcal{A},\mu))$. Towards that end we set $\mathcal{F}=\{A:\chi_{X\setminus A}\in L'\}$. Then $|\mathcal{F}|\leq |L'|$. We shall show that $\mathcal{F}$ controls the finite measure. Indeed we choose $A\in \mathcal{A}_>$. Then $B(\underline{0},A,\frac{1}{2})\cap L'\neq\emptyset$ as $\underline{0}\in \cl L'$. We choose, $\chi_{X\setminus B}\in L'\cap B(\underline{0},A,\frac{1}{2})$ for some $B\in \mathcal{A}_>$. Then $\chi_{X\setminus B}(x)<\frac{1}{2}  $ for all $x\in A$ almost everywhere($\mu$) on it $\implies$ $\chi_{X\setminus B}(x)=0$ almost everywhere($\mu$) on $\mathcal{A}\implies$ $\mu(A\setminus B)=0$ and we note that $B\in \mathcal{F}$.

\end{proof}
\begin{corollary}
The following statements are equivalent for a measure space $(X,\mathcal{A},\mu)$:
	\begin{enumerate}

		\item $\mathcal{M}_F(X,\mathcal{A},\mu)$ is {first countable}.
		\item $\mathcal{M}_F(X,\mathcal{A},\mu)$ is a {Fre\'chet space}.
		\item $\mathcal{M}_F(X,\mathcal{A},\mu)$ is a {sequential space}.
		\item $\mathcal{M}_F(X,\mathcal{A},\mu)$ is {countably tight}.
		\item $(X,\mathcal{A},\mu)$ is a {hemifinite measure space}.
		\item $\mathcal{M}_F(X,\mathcal{A},\mu)$ is of {pointwise countable type}.
	\end{enumerate}
\end{corollary}
\begin{theorem}\label{5.3}
	For any measure space $(X,\mathcal{A},\mu)$, $$w(\mathcal{M}_F(X,\mathcal{A},\mu))=cn(\mathcal{A}) d(\mathcal{M}_F(X,\mathcal{A},\mu))$$.
\end{theorem}
	\begin{proof}
	For any topological space $Y$, $d(Y)\leq w(Y)$ and $\chi(Y)\leq w(Y)$ \cite[Theorem 2.1]{j}, therefore it follows from Theorem \ref{character} that $$cn(\mathcal{A})d(\mathcal{M}_F(X,\mathcal{A},\mu))\leq w(\mathcal{M}_F(X,\mathcal{A},\mu)).$$

	 We need to prove the reverse inequality. For that purpose let, $\mathcal{F}$ be a subfamily of $\mathcal{A}_>$, which controls the finite measure and with $|\mathcal{F}|\leq cn(\mathcal{A})$ and $D$ be a dense subset of $\mathcal{M}_F(X,\mathcal{A},\mu)$ such that $|D|=d(\mathcal{M}_F(X,\mathcal{A},\mu))$. Let $L=\{B(f,F,\frac{1}{n}):f\in D, F\in \mathcal{F},n\in \mathbb{N}\}$. It suffices to show that $L$ is an open base for the topology of $\mathcal{M}_F(X,\mathcal{A},\mu)$ and hence, $w(\mathcal{M}_F(X,\mathcal{A},\mu))\leq |\mathcal{F}||D|$. Towards that assertion that $L$ is an open base for $\mathcal{M}_F(X,\mathcal{A},\mu)$, let $g\in \mathcal{M}(X,\mathcal{A},\mu)$, $H\in \mathcal{A}_>$, $\epsilon>0$. Choose a $k\in \mathbb{N}$ such that $\frac{1}{k}<\frac{\epsilon}{4}$ and an $F\in \mathcal{F}$ such that $\mu(H\setminus F)=0$. Then $B(g,F,\frac{1}{k})\subseteq  B(g,H,\epsilon)$. As $D$ is dense in $\mathcal{M}_F(X,\mathcal{A},\mu)$, we can choose a point $h\in D\cap B(g,F,\frac{1}{k})$. It is not hard to check that $g\in B(h,F,\frac{1}{k})\subseteq B(g,H,\epsilon)$. 
\end{proof}
\begin{theorem}\label{5.4}
	For any measure space $(X,\mathcal{A},\mu)$ $$d(\mathcal{M}_F(X,\mathcal{A},\mu))\leq cn(\mathcal{A}) L(\mathcal{M}_F(X,\mathcal{A},\mu)).$$
\end{theorem}
\begin{proof}
	Let $\mathcal{F}\subseteq \mathcal{A}_>$ be such that $|\mathcal{F}|\leq cn(\mathcal{A})$ and $\mathcal{F}$ controls the finite measure. Choose, $F\in \mathcal{F}$ and $m\in \mathbb{N}$. Then $\{B(f,F,\frac{1}{m}):f\in \mathcal{M}(X,\mathcal{A},\mu)\}$ is an open cover of $\mathcal{M}_F(X,\mathcal{A},\mu)$. It has a subcover $\gamma_{F,m}$ such that $|\gamma_{F,m}|=L(\mathcal{M}_F(X,\mathcal{A},\mu))$. Let $D=\{f\in \mathcal{M}(X,\mathcal{A},\mu):\text{ there exists }F\in \mathcal{F}\text{ and }m\in \mathbb{N} \text{ such that }B(f,F,\frac{1}{m})\in \gamma_{F,m}\}$. Since, $|D|\leq |\mathcal{F}|\aleph_0 L(\mathcal{M}_F(X,\mathcal{A},\mu))$, it suffices to show that $D$ is dense in $\mathcal{M}_F(X,\mathcal{A},\mu)$. Let $g\in \mathcal{M}_F(X,\mathcal{A},\mu)$, $H\in \mathcal{A}_>$, $\epsilon>0$. Since $\mathcal{F}$ controls the finite measure, there exists $F\in \mathcal{F}$ such that $\mu(H\setminus F)=0$. We find a $k\in \mathbb{N}$ such that $\frac{1}{k}<\frac{\epsilon}{2}$. Since $\gamma_{F,k}$ is a cover of $\mathcal{M}_F(X,\mathcal{A},\mu)$, there exists a member $B(h,F,\frac{1}{k})\in \gamma_{F,k}$ such that $g\in B(h,F,\frac{1}{k})$. Consequently, $h\in B(g,F,\frac{1}{k})\cap D\subseteq B(g,H,\epsilon)\cap D$. This proves the denseness of $D$ in $\mathcal{M}_F(X,\mathcal{A},\mu)$.
	
\end{proof}
\begin{theorem}\label{5.5}
	For a measure space $(X,\mathcal{A},\mu)$, $$d(\mathcal{M}_F(X,\mathcal{A},\mu))\leq cn(\mathcal{A})c(\mathcal{M}_F(X,\mathcal{A},\mu)).$$
\end{theorem}
\begin{proof}
	Let $\mathcal{F}\subseteq \mathcal{A}_>$ be a family of measurable sets, which control finite measure such that $|\mathcal{F}|=cn(\mathcal{A})$. For each $F\in \mathcal{F}$ and $k\in \mathbb{N}$, let $\mathcal{U}_F^k$ be a cellular family of basic open sets of the type $B(f,F,\frac{1}{k})$, the existence of such a family is guaranteed because we can take $\mathcal{U}_F^k=\{B(f,F,\frac{1}{k})\}$. A straight forward use of Zorn's Lemma ensures that there exists a maximal cellular family $\gamma_F^k$ for any $F\in \mathcal{F}$ and any $k\in \mathbb{N}$. Set $D=\{f\in \mathcal{M}(X,\mathcal{A},\mu):\text{ there exists }F\in \mathcal{F}\text{ and }k\in \mathbb{N}\text{ such that }B(f,F,\frac{1}{k})\in \gamma_F^k\}$. It is clear that $|D|\leq cn(\mathcal{A})c(\mathcal{M}_F(X,\mathcal{A},\mu))$. So, to complete this theorem it suffices to show that $D$ is dense in $\mathcal{M}_F(X,\mathcal{A},\mu)$. For that purpose let $g\in \mathcal{M}(X,\mathcal{A},\mu)$, $G\in \mathcal{A}_>$ and $\epsilon>0$. Since $\mathcal{F}$ controls the finite measure, there exists $F\in \mathcal{F}$ and $m\in \mathbb{N}$ such that $\mu(G\setminus F)=0$ and $\frac{1}{m}<\frac{\epsilon}{4}$. Therefore, $B(g,F,\frac{1}{m})\subseteq B(g,G,\epsilon)$. Now either $B(g,F,\frac{1}{m})\in \gamma_F^m$ or $B(g,F,\frac{1}{m})$ is outside of $\gamma_F^m$. In the second case, the maximality of $\gamma_F^m$ implies that, there exists $B(p,F,\frac{1}{m})\in \gamma_F^m$ such that $B(p,F,\frac{1}{m})\cap B(g,F,\frac{1}{m})\neq \emptyset$. Thus in any case, we can write $B(p,F,\frac{1}{m})\cap B(g,F,\frac{1}{m})\neq \emptyset$ for some $B(p,F,\frac{1}{m})\in \gamma_F^m$. We choose, $j\in B(p,F,\frac{1}{m})\cap B(g,F,\frac{1}{m}$). Then $|p-g|\leq|p-j|+|j-g|<\frac{\epsilon}{2}$ almost everywhere($\mu$) on $F$. Hence, $p\in B(g,G,\epsilon)\cap D$ (we recall that $\mu(G\setminus F)=0$).

\end{proof}
Since, the cellularity is always less than or equal to the density of a topological space, an immediate consequence of the above theorems is;
\begin{corollary}\label{5.6}
	For a measure space $(X,\mathcal{A},\mu)$, $$w(\mathcal{M}_F(X,\mathcal{A},\mu))=cn(\mathcal{A}) d(\mathcal{M}_F(X,\mathcal{A},\mu))=cn(\mathcal{A})L(\mathcal{M}_F(X,\mathcal{A},\mu))=cn(\mathcal{A})c(\mathcal{M}_F(X,\mathcal{A},\mu).)$$
\end{corollary} 

We conclude this paper with the following observations:
\begin{theorem}
	The following statements are equivalent for a measure space $(X,\mathcal{A},\mu)$:
	\begin{enumerate}
		\item $\mathcal{M}_F(X,\mathcal{A},\mu)$ is {second countable}.
		\item $\mathcal{M}_F(X,\mathcal{A},\mu)$ is {separable} and $\mu$ is a {hemifinite measure}.
		\item $\mathcal{M}_F(X,\mathcal{A},\mu)$ is {Lindel\"{o}f } and $\mu$ is a {hemifinite measure}.
		\item $\mathcal{M}_F(X,\mathcal{A},\mu)$ satisfies CCC condition ($\equiv$ countable chain condition) and $\mu$ is a {hemifinite measure}.
	\end{enumerate}
\end{theorem}
\begin{proof}
	The proof follows as combining Corollary \ref{5.6},  Theorem \ref{firstcounttable}, Theorem \ref{5.3} and Theorem \ref{5.4}.
\end{proof}

\end{document}